\documentclass[11pt]{article}	
\usepackage{amscd}
\usepackage{amsmath}
\usepackage{amsfonts}
\usepackage{amssymb}
\usepackage{graphicx}
\usepackage{url}
\usepackage{hyperref}
\usepackage{enumerate}
\usepackage{units}

\usepackage{amsthm}
\newtheorem{theorem}{Theorem}
\newtheorem{definition}[theorem]{Definition}
\newtheorem{lemma}[theorem]{Lemma}
\newtheorem{proposition}[theorem]{Proposition}
\newtheorem{corollary}[theorem]{Corollary}
\newtheorem{remark}[theorem]{Remark}

\newcommand{\ox}{\otimes}

\newcommand{\Ox}{\bigotimes}

\newcommand {\dsty}{\displaystyle}
\newcommand{\0}{{\bm 0}}
\newcommand{\1}{{\bf 1}}

\newcommand{\A}{\bm{A}}

\newcommand{\B}{\bm{B}}

\newcommand{\Bmatr}[1]{\begin{bmatrix}\displaystyle #1 \end{bmatrix}}

\newcommand{\Cc}{{\cal C}}

\newcommand{\C}{\bm{C}}

\newcommand{\D}{\bm{D}}

\newcommand{\E}{\bm{E}}

\newcommand{\Fc}{{\cal F}}
\newcommand{\F}{\bm{F}}

\renewcommand{\H}{\bm{H}}

\newcommand{\I}{\bm{I}}

\newcommand{\K}{\bm{K}}
\newcommand{\Lam}{\bms{\Lambda}}

\newcommand{\M}{\bm{M}}

\newcommand{\Pc}{{\cal P}}

\renewcommand{\P}{\bm{P}}

\newcommand{\Q}{\bm{Q}}

\newcommand{\Rc}{{\cal R}}

\renewcommand{\S}{\bm{S}}

\newcommand{\Smh}{\hat{\bm{S}}}

\newcommand{\Sm}{\bm{S}}

\newcommand{\W}{\bm{W}}
\newcommand{\X}{\bm{X}}

\newcommand{\an}[1]{\begin{align}#1\end{align}}
\newcommand{\ab}[1]{\begin{align*}#1\end{align*}}

\newcommand{\desclist}[1]{\begin{description}#1\end{description}}

\newcommand{\bms}[1]{{\boldsymbol{#1}}}
\newcommand{\bm}[1]{{\bf #1}}

\newcommand{\ch}{\hat{c}}
\newcommand{\df}[2]{\displaystyle \frac{\mbox{\rm d} #1}{\mbox{\rm d} #2}}

\newcommand{\dfinline}[2]{\mbox{\rm d} #1 / \mbox{\rm d} #2}

\newcommand{\diagMinline}[1]{{\mbox{\bf diag}\bigl[#1\bigr] } }

\newcommand{\dspfrac}[2]{\frac{\displaystyle #1}{\displaystyle #2} }

\newcommand{\enumlist}[1]{\begin{enumerate}#1\end{enumerate}}

\newcommand{\eqdef}{:=}

\newcommand{\ev}{\bm{e}}
\newcommand{\evt}{{\ev^\tp}}

\newcommand{\oldversion}[1]{}

\newcommand{\fracinline}[2]{#1/#2}

\newcommand{\goesto}{\rightarrow}

\newcommand{\hide}[1]{}	%hides text
\newcommand{\temphide}[1]{}	%hides text
\newcommand{\m}{\bm{m}}

\newcommand{\pf}[2]{\displaystyle \frac{\partial #1}{\partial #2}}

\newcommand{\pfinline}[2]{\partial #1 / \partial #2}

\newcommand{\pmu}[2]{\frac{\partial #1}{\partial  \mu_{#2}} }

\newcommand{\pr}{\protect}

\newcommand{\p}{\bm{p}}

\newcommand{\q}{\bm{q}}

\newcommand{\setd}[1]{\left\{ \; #1 \; \right\} }

\newcommand{\suchthat}{\colon}

\newcommand{\tv}{\bm{t}}
\newcommand{\tp}{\top}	% better than \sf, \tt, \rm T, or \textsf{\textsc{t}}
\newcommand{\tr}{\!^\top}

\newcommand{\uv}{\bm{u}}

\newcommand{\vv}{\bm{v}}

\newcommand{\xvh}{\hat{\x}}

\newcommand{\x}{\bm{x}}

\newcommand{\y}{\bm{y}}

\newcommand{\z}{\bm{z}}
\newfont{\gilfont}{cmsy10 scaled\magstep0}
\newcommand{\Naturals}{\mathbb{N}} %natural numbers
\newcommand{\Reals}{\mathbb{R}} % reals
\newcommand{\citet}{\cite}
\newcommand{\citep}{\cite}

\newcommand{\citeauthor}{\cite}

\newcommand{\lc}{\MakeLowercase}
\newcommand{\rlc}{\lc{r}}
\newcommand{\DL}{\D_{\sf L}}
\newcommand{\DR}{\D_{\sf R}}
\newcommand{\EH}{{\sf E_H}}
\newcommand{\EA}{{\sf E_A}}
\newcommand{\iloc}{i}

\title{
On the Ordering of Spectral Radius Product $\rlc(\A) \, \rlc (\A\D)$ Versus $\rlc(\A^2\D)$ and Related Applications
} 
\author{Lee Altenberg\\ \hide{\sf altenber@hawaii.edu}
{\footnotesize Associate Editor of \emph{BioSystems}; Ronin Institute;} \footnotesize \sf{altenber@hawaii.edu}}

\date{} 

\begin{document}
\maketitle
\begin{abstract}
For a nonnegative matrix $\A$ and real diagonal matrix $\D$, two known inequalities on the spectral radius, $r(\A^2 \D^2) \geq r(\A\D)^2$ and $r(\A) \, r(\A \D^2)  \geq r(\A\D)^2$, leave open the question of what determines the order of $r(\A^2 \D^2)$ with respect to $r(\A) \, r(\A \D^2)$.  This is a special case of a broad class of problems that arise repeatedly in ecological and evolutionary dynamics.  Here, sufficient conditions are found on $\A$ that determine orders in either direction.  For a diagonally symmetrizable nonnegative matrix $\A$ with all positive eigenvalues and nonnegative $\D$, $r(\A^2 \D) \leq r(\A) \, r(\A\D)$.  The reverse holds if all of the eigenvalues of $\A$ are negative besides the Perron root.  This is a particular case of the more general result that $r(\A [(1-m) \, r(\B) \, \I + m \B] \D)$ is monotonic in $m$ when all non-Perron eigenvalues of $\A$ have the same sign --- decreasing for positive signs and increasing for negative signs, for symmetrizable nonnegative $\A$ and $\B$ that commute.  Commuting matrices include the Kronecker products $\A, \B \in \{ \ox_{i=1}^L \M_i^{t_i} \}$, $t_i \in \{0, 1, 2, \ldots\}$, which comprise a class of application for these results.  This machinery allows analysis of the sign of
$\pfinline{}{m_j}\, r(\{\ox_{i=1}^L[(1{-}m_i) \, r(\A_i) \I_i + m_i \A_i] \}\D)$.  The eigenvalue sign conditions also provide lower or upper bounds to the harmonic mean of the expected sojourn times of Markov chains.  These inequalities appear in the asymptotic growth rates of viral quasispecies, models for the evolution of dispersal in random environments, and the evolution of site-specific mutation rates over the entire genome. \footnote{To appear in \emph{SIAM Journal on Matrix Analysis and Applications (SIMAX)}}
\\ \ \\
MSC Subject Classification: 
\desclist{\small \setlength{\itemsep}{-2pt}
\item[15A42] Inequalities involving eigenvalues and eigenvectors,
\item[15B48] Positive matrices and their generalizations, 
\item[15A18] Eigenvalues, singular values, and eigenvectors.
}	% desclist
\end{abstract}
%%%%%%%%%%%%%%%%%%%%%%%%%%%%%%%
%%%%%%%%%%%%%%%%%%%%%%%%%%%%%%%
\section{Introduction}

In a recent paper, Cohen \citet{Cohen:2013:Cauchy} compares two inequalities on the spectral radius, $r$, of products involving a nonnegative square matrix $\A$ and a positive diagonal matrix $\D$:
\an{
r(\A^2 \D^2) &\geq r(\A\D)^2 \label{eq:A2D2},\\
r(\A) \, r(\A \D^2) & \geq r(\A\D)^2. \label{eq:AAD2}
}
Inequality \eqref{eq:A2D2} is obtained in \citet{Cohen:Friedland:Kato:and:Kelly:1982}, while inequality \eqref{eq:AAD2} is obtained in \citet{Cohen:2013:Cauchy}.  The relationship between two the left-hand side expressions in the inequalities is not determined.  Cohen notes that positive matrices $\A$ and real diagonal matrices $\D$ can be chosen to give either 
\an{\label{eq:Inequalities}
r(\A^2 \D^2)> r(\A) \, r(\A \D^2) \text{\ \ or \  \ } r(\A^2 \D^2) < r(\A) \, r(\A \D^2),
}
and asks whether conditions may be found that guarantee a direction to the inequality.

This seemingly narrow question is intimately related to a very broad class of problems that arise repeatedly in ecological and evolutionary dynamics, including the evolution of site-specific mutation rates over multiple loci \citep{Altenberg:2011:Mutation}, and dispersal in random environments \citep{Altenberg:2012:Dispersal}.  Specifically, it relates to the open question \citep{Altenberg:2004:Open,Altenberg:2012:Resolvent-Positive} of what conditions on $\A$, $\B$ and $\D$ determine the sign of $\dfinline{r([(1-m) \A + m \B]\D)}{m}$.  This more general question is pursued here in answering questions about \eqref {eq:Inequalities}.

The mathematical origin of the question begins with work of the late Sam Karlin on the effect of population subdivision on the maintenance of genetic diversity.  Karlin proved a very general and important theorem on the effect of mixing on growth, which has recently been independently rediscovered \citep [Theorem 3.1, Lemma 4.2]{Kirkland:Li:and:Schreiber:2006}:
\begin{theorem}[Karlin, \protect{\citep[Theorem 5.2, pp. 117--118,194--196]{Karlin:1982}}]  
\label{Theorem:Karlin5.2}  
 Let $\P$ be an irreducible stochastic matrix, and $\D$ be a positive diagonal matrix that is not a scalar multiple of $\I$.  Then the spectral radius of $[(1-m) \I + m \P]\D$ decreases strictly in $m \in [0,1]$.
 \end{theorem}
 
The diagonal matrix $\D$ represents heterogeneous growth rates in different population subdivisions, and $\P$ represents dispersal between subdivisions.  The parameter $m$ represents the rate of mixing, which scales the transition rates $P_{ij}$ between different subdivisions.  The form $(1-m) \I + m \P$ appears in numerous models in evolution and ecology where transitions of state are caused by single events.  Models in which multiple events occur, or where there is nonuniform scaling of the transition probabilities, do not fit this form.  Often they are of the form $(1-m) \P_1 + m \P_2$.  Characterizing the relationships that make $r([(1-m) \P_1 + m \P_2]\D)$ monotonically increasing or decreasing in $m$ has the potential to solve the behavior of many of these ecological and evolutionary models.  Here we examine a partial characterization of these relationships.

The characterization relies on another insight from Karlin in the same paper \citep[Theorem 5.1, pp. 114--116, 197--198]{Karlin:1982}, that \emph{symmetrizable} stochastic matrices are analytically tractable.

%%%%%%%%
\begin{definition}[Symmetrizable Matrix \citep{Johnson:1977:Hadamard}]
A square matrix $\A$ is called {\em diagonally symmetrizable} (for brevity, \emph{symmetrizable}) to a symmetric matrix $\Sm$ if it can be represented as a product $\A = \DL \Sm \DR$, where $\DL$ and $\DR$ are positive diagonal matrices. 
\end{definition}

\begin{theorem}[Karlin, \protect{\citep[Theorem 5.1, pp. 114--116, 197--198]{Karlin:1982}}]  
\label{Theorem:Karlin5.1}  
Consider a family $\Fc$ of stochastic matrices that commute and are simultaneously symmetrizable to positive definite matrices, i.e.:
\an{
\label{eq:KarlinMhMk}
\Fc \eqdef \{ \M_i = \DL \Sm_i \DR \colon \M_h \M_k = \M_k \M_h \},
}
where $\DL$ and $\DR$ are positive diagonal matrices, and each $\Sm_h$ is a positive definite symmetric nonnegative matrix.   Let $\D$ be a positive diagonal matrix.  Then for each $\M_h, \M_k \in \Fc$:
$\label{eq:Theorem5.1}
r(\M_h \M_k \D) \leq r(\M_k \D).
$
\end{theorem}

Conditions for the inequality $r(\A^2 \D)\leq r(\A) \, r(\A \D)$ are readily obtained from Theorem \ref {Theorem:Karlin5.1}.  It is applied to \eqref{eq:Inequalities} by constraining $\D^2$ to be a positive diagonal matrix and substituting $\M_h = \M_k = \A$, and $r(\A)=1$, which yields $r(\A^2 \D^2) \leq r(\A) \, r(\A \D^2)$.

Theorem \ref{Theorem:Karlin5.1} is extended in \cite{Altenberg:2012:Dispersal} to conditions that make the spectral radius monotonic over a homotopy from $\M_k$ to $\M_h \M_k$. This monotonicity, either increasing or decreasing, establishes inequalities in each direction between $r(\A^2 \D)$ and $r(\A) \, r(\A \D)$.

%%%%%%%%%%%%%%%%%%%%%%%%%%%%%%%%%%%%%%%%
\begin{theorem}[\pr{From \cite[Theorem 33]{Altenberg:2012:Dispersal}}]\label{Theorem:Departure}
Let $\P$ and $\Q$ be transition matrices of reversible ergodic Markov chains that commute with each other.  Let $\D \neq c \, \I$ be a positive diagonal matrix, and define
\ab{
\M(m) \eqdef \P [(1{-}m)\I + m\Q], \qquad m \in [0, 1].
}
If all eigenvalues of $\P$ are positive, then $\dfinline{r(\M(m) \D)}{m}< 0$.
If all eigenvalues of $\P$ other than $\lambda_1(\P)=1$ are negative, then $\dfinline{r(\M(m) \D)}{m} > 0.$
\end{theorem}

%%%%%%%%%%
The condition in Theorem \ref{Theorem:Karlin5.1} that the matrices be symmetrizable is shown in \citet[Lemma 2]{Altenberg:2011:Mutation} to be equivalent to their being the transition matrices of reversible Markov chains.  Theorem \ref{Theorem:Departure} yields Theorem \ref{Theorem:Karlin5.1} by letting $\M_k = \P$ and $\M_h = \Q$.  Then $\M(0) = \M_k$ and $\M(1) = \M_h \M_k = \M_k \M_h$.  The hypothesis that $\M_k$ is symmetrizable to a positive definite matrix means that $\M_k$ has all positive eigenvalues, so $\dfinline{r(\M(m)\D)}{m} \leq 0$ for any positive diagonal $\D$, and thus $r(\M_h\M_k\D)=r(\M_k\M_h\D) \leq r(\M_k\D)$.  Note that the eigenvalues of $\M_h$ are irrelevant to this inequality.

For the inequality in the reverse direction, $r(\A^2 \D) \geq r(\A) \, r(\A \D)$, let all the eigenvalues of $\A$ other than $r(\A) = \lambda_1(\A)=1$ be negative and substitute $\A = \P = \Q$, so $\M(0) = \A$ and $\M(1) = \A^2$.  The result $\dfinline{r(\M(m) \D)}{m} \geq 0$ yields $r(\A^2 \D) \geq r(\A) \, r(\A\D)$ for such a stochastic symmetrizable matrix $\A$.  

In the present paper, Theorem \ref{Theorem:Departure} is generalized to all symmetrizable irreducible nonnegative matrices.  

%%%%%%%%%%
\section{Results}

The goal is to provide conditions for which the spectral radius of $\A [(1-m) \, r_B \, \I + m \B]\D$, or more generally of $[(1-m) \, \A + m \B]\D$, is monotonic in $m$.  We proceed as follows:  first, $\A$ and $\B$ are constrained to commute and be symmetrizable, which allows them to be simultaneously represented by the canonical form \eqref{eq:CanonicalForm};  second, this form is used to show that its spectral radius can be represented as a sum of squares; finally, the derivative of the spectral radius is represented as a sum of squares, and this is utilized to give conditions that determine its sign.

%%%%%%%%%%%%%%%%%%
\subsection{Preparatory Results}
The following notational conventions are used.  The elements of a matrix $\A$ are $[\A]_{ij}\equiv A_{ij}$, the columns are $[\A]_i$, and the rows are $[\A]^i$, and $\A\tr$ is its transpose.  A diagonal matrix with elements of a vector $\x$ along the diagonal is $\D_\x = \diagMinline{\x}$.  Diagonal matrix $\D$ is called \emph{nonscalar} when $\D \neq c \, \I$ for any $c \in \Reals$.  The vector with $1$ in position $i$ and zeros elsewhere is $\ev_i$. 

We review the properties of irreducible nonnegative $n \times n$ matrices.  When $\A$ is irreducible then for each $(i,j)$ there is some $t \in \Naturals$ such that $[\A^t]_{ij} > 0$.  The eigenvalues of $\A$ are represented as $\lambda_i(\A)$, $i=1, \ldots, n$, and the spectral radius by $r(\A) \eqdef \max_{i=1, \ldots, n} | \lambda_i|$.  We recall from Perron-Frobenius theory that $r(\A)$ is a simple eigenvalue of $\A$, called the \emph{Perron root}, designated here as $r_A \equiv r(\A) = \lambda_1(\A)$.  The \emph{non-Perron} eigenvalues are $\lambda_{Ai} \equiv \lambda_i(\A)$, $i = 2, \ldots, n$.  Let $\vv(\A)$ and $\uv(\A)\tr$ be the right and left \emph{Perron vectors} of $\A$, the eigenvectors associated with the Perron root, normalized so that $\evt \vv(\A) = \uv(\A)\tr \vv(\A) = 1$, where $\ev$ is the vector of ones.  Since $\A$ is irreducible, from Perron-Frobenius theory, $\vv(\A)$ and $\uv(\A)\tr$ are strictly positive and unique.    

The following canonical representation of symmetrizable matrices is used throughout.  It arises for the special case of transition matrices of reversible Markov chains \citep [p. 33]{Keilson:1979}.

\begin{lemma} [\pr{\citep[Lemma 1 (15)]{Altenberg:2011:Mutation}}]\label{Lemma:1Mutation}
Let $\A=\DL \Sm \DR$, where $\DL$ and $\DR$ are positive diagonal matrices, and $\S$ is symmetric.  Then there exists symmetric $\Smh$ with the same eigenvalues as $\A$, which are all real, where $\A= \E \Smh \E^{-1}$, and $\E=\DL^{1/2} \DR^{-1/2}$.
\end{lemma}
\begin{proof}
Direct substitution gives 
\an{
\A=\DL \Sm \DR &= \E \Smh \E^{-1} = \DL^{1/2} \DR^{-1/2} \Smh \DL^{-1/2} \DR^{1/2} \iff \notag\\
\Smh &=  \DL^{-1/2} \DR^{1/2} \DL \Sm \DR \DL^{1/2} \DR^{-1/2} 
=  \DL^{1/2} \DR^{1/2} \Sm  \DL^{1/2} \DR^{1/2} , \label{eq:Smh}
}
which is symmetric \citep[p. 82]{Fiedler:Johnson:Markham:and:Neumann:1985}.  $\Smh$ and $\A$ share the same eigenvalues since $\A \x =\E \Smh \E^{-1} \x = \lambda \x$ if and only if $\Smh \E^{-1} \x = \lambda \E^{-1} \x$.  Since $\Smh$ is symmetric, $\lambda$ must be real \citet[2.5.14 Corollary]{Horn:and:Johnson:1985}.
\end{proof}
%%%%%%%%%%%%%%%%%%%%%%%%%%%%%%%
\begin{lemma}[Canonical Form for Symmetrizable Matrices \pr{\citep[Lemma 1 (18)]{Altenberg:2011:Mutation}}]
\label{Lemma:Canonical} 
A symmetrizable matrix $\A = \DL \Sm \DR$, where $\Sm$ is symmetric and $\DL$ and $\DR$ are positive diagonal matrices, can always be put into a canonical form
\an{
\label{eq:CanonicalForm}
\A = \DL \Sm \DR = \E \K \Lam \K\tr \E^{-1},
}
where $\E=\DL^{1/2} \DR^{-1/2}$ is a positive diagonal matrix, $\K$ is an orthogonal matrix, $\Lam$ is a diagonal matrix of the eigenvalues of $\A$, the columns of $\E \K$ are right eigenvectors of $\A$, and the rows of $\K\tr \E^{-1}$ are left eigenvectors of $\A$. 
\end{lemma}
%%%%%%%
\begin{proof}
Symmetric $\Smh$ from \eqref{eq:Smh} in Lemma \ref{Lemma:1Mutation} has a symmetric Jordan canonical form $\Smh = \K \Lam \K\tr$ where $\K$ is an orthogonal matrix and $\Lam$ is a matrix of the eigenvalues of $\Smh$ \citep[4.4.7 Theorem]{Horn:and:Johnson:1985}, which by construction are also the eigenvalues of $\A$.  Hence $\A = \DL \Sm \DR = \E\Smh\E^{-1} = \E \K \Lam \K\tr \E^{-1}$.  Let $[\E\K]_i$ be the $i$th column of $\E\K$.  Then 
\ab{
\A[\E\K]_i 
&
= \E \K \Lam \K\tr \E^{-1}[\E\K]_i 
= \E \K \Lam \K\tr [\K]_i 
= \E \K \Lam \ev_i 
=  \lambda_i \E \K \ev_i 
=  \lambda_i [ \E \K]_i 
}
hence $[ \E \K]_i $ is a right eigenvector of $\A$.  The analogous derivation shows the rows of $\K\tr \E^{-1}$ to be left eigenvectors of $\A$.
\end{proof}

%%%%%%%%%%%%%%%%%%%%%%%%%%%%%%%%%%%%%%%
\begin {lemma}[Canonical Form for Commuting Symmetrizable $\A$ and $\B$]\label {Lemma:CanonicalFormAB}
%\ 
Let $\A$ and $\B$ be $n \times n$ symmetrizable irreducible nonnegative matrices that commute with each other.  Then $\A$ and $\B$ can be decomposed as
\an{
\A &= \D_{\vv}^{1/2} \D_{\uv}^{-1/2}\K \Lam_A \K\tr \D_{\vv}^{-1/2} \D_{\uv}^{1/2}, \label{eq:CanonA}\\
\B & = \D_{\vv}^{1/2} \D_{\uv}^{-1/2}\K \Lam_B \K\tr \D_{\vv}^{-1/2} \D_{\uv}^{1/2}, \label{eq:CanonB}
}
where $\vv \equiv \vv(\A)=\vv(\B)$, $\uv \equiv \uv(\A)=\uv(\B)$,  $\K$ is an orthogonal matrix, $[\K]_1 =  \D_{\vv}^{1/2} \D_{\uv}^{1/2} \ev$, and $\Lam_A$ and $\Lam_B$ are diagonal matrices of the eigenvalues of $\A$ and $\B$, respectively.
\end{lemma}
%%%%%%%
\begin{proof}
Since $\A$ and $\B$ are symmetrizable, each can be represented by canonical form \eqref{eq:CanonicalForm} which contains diagonal matrix $\Lam$ and similarity matrices $(\E\K)^{-1}= \K\tr\E^{-1}$, so $\A$ and $\B$ are diagonalizable.  Since $\A$ and $\B$ commute by hypothesis, they can be simultaneously diagonalized \citep[Theorem 1.3.19, p. 52]{Horn:and:Johnson:1985}, which means there exists an invertible $\X$ such that
$\A = \X \Lam_A \X^{-1} \text{\ and \ } \B = \X \Lam_B \X^{-1}$.  Clearly the columns of $\X$ are right eigenvectors of $\A$ and $\B$, and the rows of $\X^{-1}$ are left eigenvectors of $\A$ and $\B$, since
\ab{
\A [\X]_i = \X \Lam_A \X^{-1} [\X]_i 
=
\X \Lam_A \ev_i 
=
\X \lambda_i(\A) \ev_i 
=
\lambda_i(\A) [\X]_i , 
}
etc., so from Lemma \ref {Lemma:Canonical} we can set $\X = \E\K$ to give 
\ab{
\A &= \X \Lam_A\X^{-1} = \E\K \Lam_A (\E\K)^{-1} = \E\K \Lam_A \K\tr \E^{-1},\\
\B &= \X \Lam_B\X^{-1} = \E\K \Lam_B (\E\K)^{-1} = \E\K \Lam_B \K\tr \E^{-1}.
}

Without loss of generality, the Perron root is indexed as $\lambda_1$, so $r(\A) = \lambda_1(\A)$, $r(\B) = \lambda_1(\B)$.  Since $\A$ and $\B$ are irreducible,
\ab{
\vv  &\equiv \vv(\A) = \vv(\B) = [\E\K]_1 > \0, \\%\label{eq:vAB}
\uv\tr  &\equiv \uv(\A)\tr  = \uv(\B)\tr   = [\K\tr\E^{-1}]^1> \0 .%\label{eq:uAB} .
}
Next, $\E$ is solved in terms of $\uv\tr$ and $ \vv$: $[\E\K]_1 = \E [\K]_1 = \vv$, so $[\K]_1 = \E^{-1} \vv$,
{and} $[\K\tr\E^{-1}]^1 = [\K\tr]^1 \E^{-1} = \uv\tr$, so $[\K\tr]^1= \uv\tr  \E$, which combined give $K_{1j} = E_j^{-1} v_{j} =  u_{j} E_j$,
hence
\an{\label{eq:Ej2}
E_j^2 = \fracinline{v_{j}}{u_{j}},
}
so 
$E_j= \sqrt{\fracinline{v_{j}}{u_{j}}}$ and 
\an{\label{eq:E}
\E = \D_{\vv}^{1/2} \D_{\uv}^{-1/2}.
}
The first column of $\K$ evaluates to
\an{\label{eq:K1}
[\K]_1 = \E^{-1} \vv = \D_{\vv}^{-1/2} \D_{\uv}^{1/2} \vv
=  \D_{\vv}^{1/2} \D_{\uv}^{1/2} \ev. 
}
Here $\uv$ and $\vv$ were chosen as given, but alternatively $\E$ and $\vv$ can be chosen as given, and then $u_j = v_j / E_j^2$.
\end{proof}

%%%%%%%%%%%%%%%%%%%%%%%%%%%%%%%%%%%
\begin{theorem}[Sum-of-Squares Solution for the Spectral Radius]\label{Theorem:SpectralRadius}
\ \\
Let $\A$ and $\B$ be $n \times n$ symmetrizable irreducible nonnegative matrices that commute. Let $r_A\equiv r(\A) = \lambda_{A1}$ and $r_B \equiv r(\B) = \lambda_{B1}$ refer to their Perron roots, and $\{ \lambda_{Ai} \}$ and $\{ \lambda_{Bi} \}$ represent all of the eigenvalues of $\A$ and $\B$, respectively.  Let $\uv\tr$ and $\vv$ be the common left and right Perron vectors of $\A$ and $\B$ (Lemma \ref {Lemma:CanonicalFormAB}).  Let $\D$ be a positive diagonal matrix and define 
\ab{
\M(m) \eqdef \A [(1-m) \, r_B \, \I + m \B], \qquad m \in [0,1].
}
Let $\vv(m) \equiv \vv(\M(m) \D)$ and $\uv(m)\tr \equiv \uv(\M(m) \D)\tr$ refer to the right and left Perron vectors of $\M(m)\D$.

Then
\an{\label{eq:rho-y2}
r(\M(m) \D) &=  \sum_{i=1}^n \lambda_{Ai} \, [(1{-}m) r_B + m \lambda_{Bi} ] \, y_i(m)^2,
}
where
\an{\label{eq:yi}
y_i(m) ^2 
&
= 
\frac{1} { \dsty\sum_j D_j  \frac{ u_j}{v_j}v_j(m)^2 }\left[ \sum_j K_{ji} D_j  \left(\frac{u_j}{v_j}\right)^{1/2}  v_j(m)\right]^2,
}
and $\K$ is from the canonical form in Lemma \ref{Lemma:CanonicalFormAB}.
\end{theorem}
%%%%%%%
\begin{proof} 
One can represent $\M(m)$ using the canonical forms \eqref{eq:CanonA}, \eqref{eq:CanonB}:
\an{ \label{eq:CanonM}
\M(m) &= \D_{\vv}^{1/2} \D_{\uv}^{-1/2} \K \Lam_A[(1{-}m) \, r_B \, \I + m \Lam_B] \K\tr \D_{\vv}^{-1/2} \D_{\uv}^{1/2}.
}
This form will be used to produce a symmetric matrix similar to $\M(m) \D$, which allows use of the Rayleigh-Ritz variational formula for the spectral radius.    The expression will be seen to simplify to the sum of squared terms. 

For brevity, (recalling $\E = \D_{\vv}^{1/2} \D_{\uv}^{-1/2}$ from \eqref{eq:E}) define the symmetric matrices:
\an{
\H_m &\eqdef \K \Lam_A [(1{-}m) r_B \I  + m \Lam_B] \K\tr,  \text{\ and}\label{eq:Hm}\\
\S_m &\eqdef \D^{1/2} \H_m \D^{1/2} 
=  \D^{1/2} \K \Lam_A[(1-m) r_B \I + m \Lam_B] \K\tr \D^{1/2} \label{eq:S}
\\
&
=  \D^{1/2} \E^{-1} \M(m) \E \D^{1/2}. \notag
}
Thus $\M(m) = \E \H_m \E^{-1}$.  Since $\M(m) \geq \0$ is irreducible, and $\E$ and $\D$ are positive diagonal matrices, then $\H_m, \S_m \geq \0$ are irreducible.  The following identities are obtained:
\ab{
r(\M(m) \D) &= r(  \E \H_m \E^{-1} \D ) = r(  \H_m \E^{-1} \D \E) 
= r( \H_m \D )  = r(\D^{1/2} \H_m \D^{1/2} ) 
 \\
&
= r(\S_m).
}
Since $\S_m$ is symmetric, we may apply the Rayleigh-Ritz variational formula for the spectral radius \citep[Theorem 4.2.2, p. 176]{Horn:and:Johnson:1985}:
\an{\label{eq:RayleighRitz}
r(\S_m) = \max_{\x \neq \0} \frac{\x\tr \S_m \x}{\x\tr \x}.
}
This yields
\an{\label{eq:xVariational}
r(\M(m) \D) & = r(\S_m) = \max_{\x\tr \x = 1} \x\tr \S_m \x \notag\\
&=\max_{\x\tr \x = 1}  \x\tr \D^{1/2}  \K \Lam_A [(1-m) r_B \I + m \Lam_B] \K\tr \D^{1/2} \x.
}
Any $\xvh$ that yields the maximum in \eqref{eq:xVariational} is an eigenvector of $\S_m$ \citep[p. 33]{Gould:1966:Variational}.  Since $\S_m \geq \0$ is irreducible, by Perron-Frobenius theory, $\xvh(m) > 0$ is therefore the unique left and right Perron vector of $\S_m$.  This allows one to write
\an{\label{eq:xhat}
r(\M(m) \D) %\notag \\ &
&=  \xvh(m)\tr \D^{1/2}  \K \Lam_A [(1-m) r_B \I + m \Lam_B] \K\tr \D^{1/2} \, \xvh(m).
}
Define 
\an{\label{eq:ydef}
\y(m) \eqdef \K\tr  \D^{1/2}  \xvh(m).
}
Substitution of \eqref{eq:ydef} into \eqref{eq:xhat} yields \eqref{eq:rho-y2}:
\ab{
r(\M(m) \D) &=  \xvh(m)\tr \D^{1/2}  \K \Lam_A [(1{-}m) r_B \I + m \Lam_B] \K\tr \D^{1/2} \, \xvh(m) \notag \\
&= \y(m)\tr  \Lam_A [(1{-}m)  r_B \I + m \Lam_B] \, \y(m) \notag \\
&= \sum_{i=1}^n \lambda_{Ai} [(1{-}m)  r_B + m \lambda_{Bi} ] \, y_i(m)^2 . 
}
Next, $\y(m)$ will be solved in terms of $\vv(m)$ and $\uv(m)$ by solving for $\xvh(m)$, using the following two facts.  For brevity, define $\Lam_m \eqdef  \Lam_A [(1-m) r_B \I + m \Lam_B]$, so $ \M(m) =  \E \K  \Lam_m \K\tr \E^{-1}$:
\an{
1.& \ r(\M(m)\D) \ \vv(m) = \M(m) \D \, \vv(m)
=  \E \K  \Lam_m \K\tr \E^{-1} \D \, \vv(m);  \label{fact:v} \\
2.& \ r(\M(m) \D) \ \xvh(m) =  \D^{1/2}  \K  \Lam_m \K\tr \D^{1/2} \, \xvh(m). \label{fact:x}
}

Multiplication on the left by $\E \D^{-1/2}$ on both sides of \eqref{fact:x} reveals the right Perron vector $\vv(m) = \vv(\M(m)\D)$:
\an{
r(\M(m) \D) \ ( \E \D^{-1/2}) \,\xvh(m)  
&= ( \E \D^{-1/2}) \D^{1/2}  \K \Lam_m  \K\tr  \D^{1/2} \, \xvh(m) \notag \\
&= \E  \K \Lam_m \K\tr ( \E^{-1} \D \E \D^{-1}) \D^{1/2} \, \xvh(m) \notag \\
&= ( \E  \K \Lam_m \K\tr \E^{-1} \D ) ( \E \D^{-1/2} \xvh(m)) \notag \\
&= \M(m) \D \ \ ( \E \D^{-1/2} \xvh(m)) , \label{eq:eigvecDerivation}
}
which shows that $\E \D^{-1/2} \xvh(m)$ is the right Perron vector of $\M(m) \D$, unique up to scaling, i.e.
\ab{
\vv(m) = \ch(m) \, \E \D^{-1/2} \xvh(m)
=  \ch(m) \, \D_{\vv}^{1/2} \D_{\uv}^{-1/2} \D^{-1/2} \xvh(m) ,
}
for some $\ch(m)$ to be solved, which gives
\an{\label{eq:x=DDv}
\xvh(m) =  \frac{1}{\ch(m)} \D_{\vv}^{-1/2} \D_{\uv}^{1/2}  \D^{1/2}  \ \vv(m) .
}
The normalization constraint $\xvh(m)\tr \xvh(m) = 1$ yields
\ab{
1 &= \xvh(m)\tr \xvh(m) =  \frac{1}{\ch(m)^2} \vv(m)\tr \D_{\vv}^{-1} \D_{\uv}  \D  \, \vv(m),
}
so
\an{\label{eq:ch}
\ch(m) &= \sqrt{\vv(m)\tr \, \D_{\vv}^{-1} \D_{\uv} \D \, \vv(m) }.
}
Substitution for $\xvh(m)$ now produces \eqref{eq:yi}:
\an{
\y(m) &\eqdef \K\tr  \D^{1/2}  \xvh(m)
= \K\tr  \D^{1/2}  \frac{1}{\ch(m)} \D_{\vv}^{-1/2} \D_{\uv}^{1/2}  \D^{1/2}  \ \vv(m) \notag
\\
&
=  \frac{1}{\sqrt{\vv(m)\tr \ \D_{\vv}^{-1} \D_{\uv} \D \, \vv(m) }} \ \K\tr  \D_{\vv}^{-1/2} \D_{\uv}^{1/2}  \D  \, \vv(m).  \label{eq:yvec}  
}
Each element of $\y(m)$ is thus
\ab{
y_i(m) 
&
= 
\frac{1} {\sqrt{ \dsty\sum_j D_j \frac{u_j}{v_j} v_j(m)^2 } } \sum_j K_{ji} D_j  \frac{u_j^{1/2}}{v_j^{1/2} }v_j(m),\\
&
= 
\frac{1} {\sqrt{ \dsty\sum_j D_j  \left( \frac{v_j(m)}{E_j}\right)^2} } \sum_j K_{ji} D_j  \frac{ v_j(m)}{E_j} ,
}
the last equality using \eqref{eq:Ej2}, $v_j / u_j = E_j^{2}$, which shows the role of terms $v_j(m)/{E_j}$.
Hence
\ab{
y_i(m) ^2 
&
= 
\frac{1} { \dsty\sum_j D_j  \frac{ u_j}{v_j} v_j(m)^2}\left( \sum_j K_{ji} D_j  \left( \frac{u_j}{v_j}\right)^{1/2} v_j(m)\right)^2.
}
\end{proof} % {Theorem:SpectralRadius}

%%%%%%%%%%%%%
\subsection{Main Results}

The general open question is to analyze $r( (1-m) \A + m \B)$ as $m$ is varied.  For an arbitrary irreducible nonnegative matrix $\F(m)$ that is a differentiable function of $m$, the derivative of its spectral radius follows the general relation \citep[Sec. 9.1.1]{Caswell:2000}:
\an{\label{eq:Derivative}
\df{}{m} r(\F(m)) = \uv(\F(m))\tr \df{\F(m)}{m} \vv(\F(m)).
}
The derivatives of $\uv(\F(m))$ and $\vv(\F(m))$ do not appear in \eqref{eq:Derivative} because they are critical points with respect to $r(\F(m))$ \citep{Friedland:and:Karlin:1975}.  From \eqref {eq:Derivative}, therefore, one has the general result that 
\ab{
\df{}{m}r( (1-m) \A + m \B) = \uv( (1-m) \A + m \B)\tr \, (\B - \A)\, \vv ( (1-m) \A + m \B),
}
but this is not very informative.  With the structures introduced in the preparatory results above, more specific results can be provided.

%%%%%%%%%%%%%%%%%%%%%%%%%%
\begin{theorem}[Main Result]\label{Theorem:Main}
Let $\A$ and $\B$ be $n \times n$ symmetrizable irreducible nonnegative matrices that commute with each other, with Perron roots $r_A\equiv r(\A) = \lambda_{A1}$ and $r_B \equiv r(\B) = \lambda_{B1}$, and common left and right Perron vectors, $\uv\tr$ and $\vv$.  Let $\D$ be a nonscalar positive diagonal matrix, and suppose
\an{\label{eq:M}
\M(m) \eqdef \A [(1{-}m) \, r_B \, \I + m\B], \qquad m \in [0,1].
}
\enumlist{
\item[C1.] If all eigenvalues of $\A$ are positive, then 
\an{\label{eq:EigPos}
\df{}{m}r(\M(m) \D)  < 0.
}
\item[C2.] If all eigenvalues of $\A$ other than $r_A = \lambda_1(\A)$ are negative, then \label{item:EigNeg}
\an{\label{eq:EigNeg}
\df{}{m}r(\M(m) \D)   > 0.
}
\item[C3.]   If $\lambda_{Ai} = 0$ for all $i \in \{2, \ldots, n\}$, then $\dfinline{r(\M(m) \D)}{m} = 0$.
\item[C4.]  If C1 or C2 hold except for some $i \in \{2, \ldots, n\}$ for which $\lambda_{Ai} = 0$, then inequalities \eqref{eq:EigPos} and \eqref{eq:EigNeg} are replaced by non-strict inequalities. 
}	% enumlist
\end{theorem}  
%%%%
\begin{proof}
The sum-of-squares form in Theorem \ref{Theorem:SpectralRadius} is now utilized to analyze the derivative of the spectral radius.    Application of \eqref{eq:Derivative} gives
\ab{
\df{}{m} r(\M(m) \D) %\notag \\ &
&=  \xvh(m)\tr \D^{1/2}  \K \Lam_A \df{ [(1-m) r_B \I + m \Lam_B]}{m} \K\tr \D^{1/2} \, \xvh(m)
\\
&
=
\xvh(m)\tr \D^{1/2}  \K \Lam_A ( \Lam_B-r_B \I ) \K\tr \D^{1/2} \, \xvh(m).
}
Substitution with $\y(m) \eqdef \K\tr  \D^{1/2} \, \xvh(m)$ yields:
\an{\label{eq:y2}
\df{ }{m} r(\M(m) \D)
&= \y(m)\tr  \Lam_A ( \Lam_B - r_B \I ) \, \y(m)% \notag \\
%&
= \sum_{i=1}^n \lambda_{Ai} \, (\lambda_{Bi} - r_B) \, y_i(m)^2.
}

%%%%%%%%%%%%%%%%%%%%%%%%%%%%%%%%%%%% resume proofreading
We know the following about the terms in the sum in \eqref{eq:y2}:
\enumlist{
\item $\lambda_{B1} - r_B = 0$.  Thus the first term $i=1$ of the sum is zero.  \label{item:lambda0}
%%%%
\item \emph{For $i \in \{2, \ldots, n\}$, $\lambda_{Bi} - r_B < 0$, hence $(\lambda_{Bi} - r_B) y_i^2 \leq 0$}.  Since $\B$ is symmetrizable, $\lambda_{Bi} \in \Reals$.  Since $\B$ is irreducible the Perron root has multiplicity $1$, and $| \lambda_{Bi} | \leq r_B$ \citet [Theorems 1.1, 1.5]{Seneta:2006}. Together these imply $\lambda_{Bi}  < r_B$ for $i \in \{2, \ldots, n\}$. \label{item:lambdaneg}
%%%%
\item \emph{\label{item:y not 0}$y_i(m) \neq 0$ for at least one $i \in \{2, \ldots, n\}$, whenever $\D \neq c \, \I$ for any $c > 0$.}   \label{item:ypos}  
Suppose to the contrary that $y_i(m) = 0$ for all $i \in \{2, \ldots, n\}$.  That means $\y(m) = y_1 (m) \, \ev_1$ so \eqref{eq:yvec} becomes
\an{\label{eq:ym}
\y(m) =\ y_1(m) \, \ev_1 =  \ch(m)^{-1} \ \K\tr  \D_{\vv}^{-1/2} \D_{\uv}^{1/2}  \D  \, \vv.
}

Now multiply on the left by nonsingular $\D_{\vv}^{1/2} \D_{\uv}^{-1/2} \K$, and use $[\K]_1 = \D_{\vv}^{1/2} \D_{\uv}^{1/2} \ev$ \eqref{eq:K1}:
\an{
y_1(m) \, (\D_{\vv}^{1/2} \D_{\uv}^{-1/2} \K) \,\ev_1 
&= y_1(m) \, \D_{\vv}^{1/2} \D_{\uv}^{-1/2} [\K]_1 \notag
\\
&
= y_1(m) \, \D_{\vv}^{1/2} \D_{\uv}^{-1/2} \D_{\vv}^{1/2} \D_{\uv}^{1/2} \ev \notag
\\
&
= y_1(m) \,   \vv ,  \label{eq:y1v}
\\
\ch(m)^{-1}  (\D_{\vv}^{1/2} \D_{\uv}^{-1/2} \K)  \K\tr  &\D_{\vv}^{-1/2} \D_{\uv}^{1/2}  \D  \, \vv
=
 \ch(m)^{-1} \  \D  \, \vv. \label{eq:cDv}
}
Since \eqref{eq:y1v} and \eqref{eq:cDv} are equal by \eqref{eq:ym}, then $\vv = (y_1(m)/\ch(m)) \D \vv > \0$, which implies $\D = (\ch(m)/y_1(m)) \I$, contrary to hypothesis that $\D \neq c \, \I$ for any $c > 0$.  Thus $\D$ being nonscalar implies that $y_i(m) \neq 0$ for at least one $i \in \{2, \ldots, n\}$. 
}

Combining points \ref{item:lambdaneg}., and \ref{item:ypos}.\ above, we have $(\lambda_{Bi} - r_B) y_i(m)^2 < 0$ for at least one $i \in \{2, \ldots, n\}$, while from point \ref{item:lambda0}., $\lambda_{A1}(\lambda_{B1} - r_B) y_1(m)^2 = 0$.  Thus, if the signs of $\lambda_{Ai}$, $i=2, \ldots, n$ are the same, then the nonzero terms in the sum in \eqref{eq:y2} all have the same sign, opposite from $\lambda_{Ai}$, and there is at least one such nonzero term. Therefore,
\enumlist{
\item if $\lambda_{Ai} > 0$ for all $i$, then $\lambda_{Ai} (\lambda_{Bi} -  r_B) < 0 \ \forall i \geq 2$, thus
\ab{
\df{ }{m}r(\M(m) \D) = \sum_{i=2}^n \lambda_{Ai} (\lambda_{Bi} -  r_B) y_i^2(m) <  0;
}
\item if $\lambda_{Ai} < 0$ for $i = 2, \ldots, n$, then $\lambda_{Ai} (\lambda_{Bi} - r_B) > 0 \ \forall i \geq 2$, thus
\ab{
\df{ }{m}r(\M(m) \D) = \sum_{i=2}^n \lambda_{Ai} (\lambda_{Bi} -  r_B) y_i^2(m) >  0;
}
\item if $\lambda_{Ai} = 0$ for all $i \in \{2, \ldots, n\}$, then all the terms in \eqref{eq:y2} are zero so $\dfinline{r(\M(m) \D)}{m} = 0$;
\item 
if $\lambda_{Ai} = 0$ for some $i \in \{2, \ldots, n\}$, we cannot exclude the possibility that the one necessary nonzero value among $y_2(m)$, \ldots, $y_n(m)$ happens to be $y_i(m)$, while $y_j(m) = 0$ for all $j \notin \{i, 1\}$, in which case all the terms in \eqref{eq:y2} would be zero.  In this case the inequalities in \eqref{eq:EigPos} and \eqref{eq:EigNeg} cannot be guaranteed to be strict.
}	% enumlist
If the non-Perron eigenvalues of $\A$ are a mix of positive, negative, or zero values, there may be positive, negative,  or zero terms $\lambda_{Ai} (\lambda_{Bi} - r_B) y_i^2$ for $i=2, \ldots, n$, so the sign of $\dfinline{r(\M(m) \D)}{m}$ depends on the particular magnitudes of $\lambda_{Ai}$, $\lambda_{Bi}$, $r_B$, and $y_i(m)$.
\end{proof}

%%%%%%%%%%%%%%%%%%%%%%%%%%%%%%
\begin{theorem}[Main Result Extension]\label{Theorem:MainAB}
Let $\A$ and $\B$ be $n \times n$ symmetrizable irreducible nonnegative matrices that commute with each other, with equal Perron roots $r(\A) = r(\B) = \lambda_{A1}= \lambda_{B1}$, and common left and right Perron vectors, $\uv\tr$ and $\vv$.  Let $\lambda_{Ai}$ and $\lambda_{Bi}$, $i = 2, \ldots, n$ be the non-Perron eigenvalues.  Let $\D$ be a nonscalar positive diagonal matrix, and suppose
\ab{
\M(m) \eqdef (1{-}m)  \A + m\B, \qquad m \in [0,1].
}
\enumlist{
\item If $\lambda_{Ai} > \lambda_{Bi}$ for $i=2, \ldots n$, then $\dfinline{r(\M(m) \D)}{m}  < 0$ and $r(\A\D) > r(\B\D)$. \label{item:<}
\item If $\lambda_{Ai} < \lambda_{Bi}$ for $i=2, \ldots n$, then $\dfinline{r(\M(m) \D)}{m} > 0$ and $r(\A\D) < r(\B\D)$. \label{item:>}
\item If $\lambda_{Ai} = \lambda_{Bi}$ for $i=2, \ldots n$, then $\dfinline{r(\M(m) \D)}{m} = 0$ and $r(\A\D) = r(\B\D)$.
\item If $\lambda_{Ai} = \lambda_{Bi}$ for at least one $i \in \{2, \ldots, n\}$, then the inequalities in \ref{item:<} and \ref{item:>} are replaced by non-strict inequalities.
}
\end{theorem}  
%%%%
\begin{proof}
The proof follows that of Theorem \ref{Theorem:Main} with some substitutions.  $\M(m)$ has the canonical representation
\ab{
\M(m) \eqdef (1{-}m)  \A + m\B
= \D_{\vv}^{1/2} \D_{\uv}^{-1/2} \K [(1{-}m) \, \Lam_A + m \Lam_B] \K\tr \D_{\vv}^{-1/2} \D_{\uv}^{1/2}.
}
The spectral radius has the sum-of-squares form as developed in \eqref {eq:S}--\eqref{eq:yvec}, where $\xvh(m)$ is as given in \eqref {eq:x=DDv},  and the derivative of the spectral radius evaluates to
\ab{
\df{}{m} r(\M(m) \D)
&=  \xvh(m)\tr \D^{1/2}  \K  \df{ [(1-m) \Lam_A+ m \Lam_B]}{m} \K\tr \D^{1/2} \, \xvh(m)
\\
&
=
\xvh(m)\tr \D^{1/2} \K( \Lam_B- \Lam_A ) \K\tr \D^{1/2} \, \xvh(m).
}
Substitution with $\y(m) \eqdef \K\tr  \D^{1/2} \, \xvh(m)$ yields:
\an{\label{eq:ABy2}
\df{ }{m} r(\M(m) \D)
&= \y(m)\tr   ( \Lam_B - \Lam_A ) \, \y(m) = \sum_{i=1}^n  \, (\lambda_{Bi} -\lambda_{Ai}) \, y_i(m)^2.
}

%%%%%%%%%%%%%%%%
The relevant facts about \eqref{eq:ABy2} are:
\enumlist{
\item $\lambda_{B1} - \lambda_{A1} = 0$ by construction.  Thus the first term $i=1$ of the sum is zero. 
%%%%
\item $y_i(m) \neq 0$ for at least one $i \in \{2, \ldots, n\}$, whenever $\D \neq c \, \I$ for any $c > 0$, as in Theorem \ref{Theorem:Main}.   
}	% enumlist

If $\lambda_{Ai} > \lambda_{Bi}$ for $i=2,\ldots, n$ then all of the terms $ (\lambda_{Bi} -\lambda_{Ai}) \, y_i(m)^2$ in \eqref {eq:ABy2} are nonpositive, and at least one is negative, therefore $\dfinline{ r(\M(m)\D)}{m}$ is negative.  If $\lambda_{Ai} < \lambda_{Bi}$ for $i=2,\ldots, n$ then all of the terms in \eqref {eq:ABy2} are nonnegative, and at least one is positive, therefore $\dfinline{ r(\M(m)\D)}{m}$ is positive.   If $\lambda_{Ai} = \lambda_{Bi}$ for $i=2,\ldots, n$ all of the terms in \eqref {eq:ABy2} are zero so $\dfinline{ r(\M(m)\D)}{m}=0$.

As in Theorem \ref{Theorem:Main}, if $\lambda_{Ai} = \lambda_{Bi}$ for some $i \in \{2, \ldots, n\}$, we cannot exclude the possibility that the necessary nonzero value among $y_2, \ldots, y_n$ happens to be $y_i$, while $y_j = 0$ for all $j \notin \{1, i\}$, in which case all of the terms in \eqref {eq:ABy2} are zero so $\dfinline{ r(\M(m)\D)}{m}=0$.  Thus inequalities in \ref {item:<} and \ref {item:>} are not guaranteed to be strict if there is a single equality between non-Perron eigenvalues of $\A$ and $\B$.  
\end{proof}
\begin{remark}\rm
It is notable here that the relation on the Perron root, $r(\A\D) > r(\B\D)$, for $\A$ and $\B$ with the same Perron root, occurs when the \emph{non}-Perron eigenvalues of $\A$ dominate those of $\B$.  Domination here means $\lambda_{Ai} > \lambda_{Bi}$ where $i$ is the index on the non-Perron eigenvectors as ordered in $\K$.  The tempting question for generalization is whether this relation holds if the orthogonality condition on $\K$ is relaxed.\end{remark}

%%%%%%%%%%%%%%%%%
\begin{corollary}\label{Corollary:n=2}
For the case $n=2$ of Theorem \ref{Theorem:Main}, $\dfinline{r(\A [(1{-}m) \, r_B \, \I + m\B]\D)}{m}$ has the opposite sign of $\lambda_{A2}$, and is zero if $\lambda_{A2}=0$.
\end{corollary}
\begin{proof}
When $n=2$, $\A$ has only one eigenvalue besides the Perron root.  Therefore $\D \neq c \, \I$ implies $y_2(m) \neq 0$ in \eqref{eq:y2}, so $(\lambda_{B2}-r_B) y_2^2 < 0$.  Thus
\ab{
\df{}{m} r (\A [(1{-}m) \, r_B \, \I + m\B]\D) 
&
= \lambda_{A1}(\lambda_{B1}-r_B) y_1^2 + \lambda_{A2}(\lambda_{B2}-r_B) y_2^2 
\\
&
= r_A * 0 * y_1^2 + \lambda_{A2} (\lambda_{B2}-r_B) y_2^2.
} 
Therefore $\dfinline{ r(\M(m) \D)}{m}$ has the opposite sign of $\lambda_{A2}$, or is $0$ if $\lambda_{A2}=0$.
\end{proof}

The following is immediate:
\begin{corollary} 
In Theorem \ref{Theorem:Main}, the term $r_B \I$ in $\A [(1{-}m) \, r_B \, \I + m\B]\D$ may be replaced by any symmetrizable nonnegative matrix $\C$ that commutes with $\A$ and $\B$ for which $r(\C) = r(\B)$ and $\lambda_{Ci} > \lambda_{Bi}$, $i=2, \ldots, n$. 
\end{corollary}

Note that the indices $i \in \{2, \ldots, n\}$ are not ordered here by the size of the eigenvalues as is commonly done, but are set by the arbitrary indexing of the non-Perron eigenvectors.

\begin{remark}\rm
Karlin \citep[p. 645]{Karlin:1976} \citep[p. 116]{Karlin:1982} inexplicably asserted that for $\P = \Bmatr{0&1\\1&0}$, $\M(m)=(1-m) \I + m \P$, and $\D$ nonscalar, $r(\M(m)\D)$ decreases in $m$ for $m\in [0,1/2]$, and increases for $m \in [1/2, 1]$ over which $(1{-}m) \I + m \P$ loses its positive definiteness since $\lambda_2(\M(m)) = 1 - 2m$.  His own Theorem 5.2 \citep{Karlin:1982} (Theorem \ref {Theorem:Karlin5.2} here), however, shows that $r ([(1{-}m) \I + m \P]\D) $ decreases in $m$ for all $m \in [0, 1]$.  Karlin correctly intuited that \emph{something} reverses when $\M(m)$ loses positive definiteness (thus departing from the condition of Theorem \ref{Theorem:Karlin5.1}), but the form was wrong.  Perhaps the form he sensed was $\M(m)^2 \D$:
\end{remark}
\begin{corollary} \label {Corollary:M2D}
Let $\P = \Bmatr{0&1\\1&0}$, and $\D$ be a nonscalar positive diagonal matrix.  Then
$\dfinline{r([(1-m) \I + m \P]^2 \D)}{m} 
$ decreases in $m$ when $m \in [0,1/2]$ and increases in $m$ when $m \in [1/2,1]$.  
\end{corollary}
%%%%
\begin{proof}
Since $\lambda_2((1-m) \I + m \P) = 1-2m> 0$ for $m \in [0, 1/2)$ and $1-2m< 0$ for $m \in (1/2,1]$,  Corollary \ref{Corollary:n=2} gives us that
\ab{
r([(1-m_1) \I + m_1\P] [(1-m_2) \I + m_2 \P]\D)
}
decreases in $m_2$ when $m_1 \in [0,1/2]$ and increases in $m_2$ when $m_1 \in [1/2,1]$, and similarly when $m_1$ and $m_2$ are interchanged.  Writing $\M(m) = (1-m) \I + m \P$, when $m_1, m_2 \in [0,1/2]$ then $r(\M(m_1) \M(m_2)\D)$ decreases in both $m_1$ and $m_2$, and when $m_1, m_2 \in [1/2, 1]$ then $r(\M(m_1) \M(m_2)\D)$ increases in both $m_1$ and $m_2$.  Setting $m_1=m_2=m$ completes the proof.
\end{proof}

%%%%%%%%%%%%%%%%%%%%%%%%%%%%%
Finally, we return to the question posed in the beginning.
\begin{corollary}[Partial answer to Cohen's open question]\label{Corollary:A2D}
Let $\A$ be an $n \times n$ symmetrizable nonnegative matrix and $\D$ be a positive diagonal matrix.  If all of the eigenvalues of $\A$ are positive, then $ r(\A) \, r(\A\D) \geq r(\A^2 \D)$.  If all of the non-Perron eigenvalues of $\A$ are negative, then $ r(\A) \, r(\A\D) \leq r(\A^2 \D)$.  When $\A$ is irreducible and $\D$ is nonscalar, then the above inequalities are strict.
If all of the non-Perron eigenvalues of $\A$ are zero, then $ r(\A) \, r(\A\D) = r(\A^2 \D)$.
\end{corollary}
%%%%%%%
\begin{proof}
Let $\A$ be irreducible and $\D$ nonscalar.  Apply Theorem \ref{Theorem:Main} with $\A= \B$.  Then $\M(0) = r(\A) \A$ and $\M(1) = \A^2$.  If all the eigenvalues of $\A$ are positive, then by Theorem \ref{Theorem:Main}, $\dfinline{r(\M(m)\D)}{m} < 0$, so $r(\M(0)\D) = r(\A) \, r(\A\D) > r(\A^2 \D) = r(\M(1)\D)$.  If all the non-Perron eigenvalues of $\A$ are negative, then $\dfinline{r(\M(m)\D)}{m} > 0$, so $r(\M(0)\D) = r(\A) \, r(\A\D) < r(\A^2 \D) = r(\M(1)\D)$.  If $\lambda_{Ai}=0$ for $i \in \{2, \ldots, n\}$ then $\dfinline{r(\M(m)\D)}{m} = 0$ so $r(\A) \, r(\A\D) = r(\A^2 \D)$.

Now, let $\D = c \, \I$ for $c > 0$.  Then  $r(\A) \, r(\A\D) = c \, r(\A)^2 = c\, r(\A^2)$ so equality holds.  
Let $\A$ be reducible. A reducible symmetrizable nonnegative matrix $\A$ is always the limit of some sequence of symmetrizable irreducible nonnegative matrices, for which the eigenvalues remain on the real line.  If $\lambda_{Ai}, i=2, \ldots, n$ are all negative or all positive, then they continue to be so for these perturbations of $\A$ by the continuity of the eigenvalues.  For each perturbation, the sign of $\dfinline{r(\M(m)\D)}{m}$ is maintained, but in the limit equality cannot be excluded, so only the non-strict versions of the inequalities are assured for reducible matrices.
\end{proof}

%%%%%%%%%%%%%%%%%%
\subsection{Levinger's Theorem}
Mention should be made of a well-known special case of the general open question about variation in $r((1-m) \A + m \B)$, Levinger's Theorem \citep{Levinger:1970}, where $\B = \A\tr$.  Levinger found that $r((1-m) \A + m \A\tr)$ is nonincreasing on $m \in [0, 1/2]$ and nondecreasing on $m \in [1/2, 1]$.  This was generalized in \citet[Theorem 7]{Alpin:and:Kolotilina:1998} to $r((1-m) \A + m \C^{-1}\A\tr \C)$, where $\C$ is a positive diagonal matrix.  Fiedler \citet{Fiedler:1995:Numerical} showed that $r((1-m) \A + m \A\tr)$ is concave in $m$ for some interval with midpoint $m=1/2$.  

The relation between Levinger's Theorem and Theorem \ref {Theorem:MainAB} is that they are in a sense orthogonal, in that their conditions intersect only in the trivial case of $(\C\A)\tr = \C\A$.  Letting the two forms intersect gives $(1-m) \A_1 + m \C^{-1}\A_1\tr \C = [(1-m)\A + m \B] \D$, where $\A = \A_1 \D^{-1}$, and $\B = \C^{-1}\A_1\tr \C \D^{-1}$.  The assumption of simultaneous symmetrizability means $\A=\E \Smh_A \E^{-1}$ and $\B=\E \Smh_B \E^{-1}$, so $\A_1  = \E \Smh_A \E^{-1} \D $ and
\ab{
\B &= \C^{-1}\A_1\tr \C \D^{-1} = \C^{-1} (\E^{-1}\D  \Smh_A \E ) \C \D^{-1} = \E \Smh_B \E^{-1} \iff \\
\Smh_B &=  \E^{-1}(\C^{-1} \E^{-1}\D  \Smh_A \E  \C \D^{-1}) \E   \iff \E^{-2}\C^{-1}\D = \E^2  \C \D^{-1}    \iff  \\
 \I&= \E^2  \C \D^{-1}   \iff \Smh_B = \Smh_A \iff \A = \B =  \A_1 \D^{-1} = \C^{-1}\A_1\tr \C \D^{-1} \iff\\
\A_1 &= \C^{-1}\A_1\tr \C \iff \C \A_1 = \A_1\tr\C = (\C \A_1)\tr.
}
This ``orthogonality'' between Levinger's Theorem and Theorem \ref {Theorem:MainAB} opens the question of whether results could be obtained on a space of variation spanned by the forms of variation from Levinger's Theorem and Theorem \ref {Theorem:MainAB}, but this is not pursued here.

\begin{remark}\rm
In reviewing the literature on Levinger's Theorem, a number of overlaps are noted.  The elementary proof of Levinger's Theorem in \citet{Fiedler:1995:Numerical} defines `balanced' matrices, which is the same as `sum-symmetric' introduced in \citet{Afriat:1974} and `line-sum-symmetric' from \citet{Eaves:Hoffman:Rothblum:and:Schneider:1985}.  Lemmas 2.1 and 2.2 in \citet{Fiedler:1995:Numerical} correspond to Corollaries 3 and 5 in \citep{Eaves:Hoffman:Rothblum:and:Schneider:1985}, but the proofs are quite distinct.

A multiplicative version of Levinger's theorem is given in \citet{Alpin:and:Kolotilina:1998}, which utilizes the weighted geometric mean Lemmas 1 and 3 from \citet{Elsner:Johnson:and:DaSilva:1988}, that $r(\A^{(m)} \circ \B^{(1-m)} ) \leq r(\A)^m \, r(\B)^{1-m}$, where $[\A^{(m)}]_{ij} \equiv A_{ij}^m$ and $\A \circ \B$ is the Schur-Hadamard product.  These lemmas are contained within Nussbaum's  omnibus Theorem 1.1 \citet {Nussbaum:1986:Convexity}, as excerpted in \cite[Theorem 13]{Altenberg:2013:Sharpened}, and the proofs all rely on H\"older's Inequality.
\end{remark}

%%%%%%%%%%%%%%%%%%%%%%%%%%%%

%%%%%%%%%%%%
\section{Applications}

The inequalities examined here arise naturally in models of population dynamics.  Karlin derived Theorems \ref{Theorem:Karlin5.2} \citep[Theorem 5.2]{Karlin:1982} and \ref{Theorem:Karlin5.1} \citep[Theorem 5.1]{Karlin:1982} in order to analyze the protection of genetic diversity in a subdivided population where $\M$ is the matrix of dispersal probabilities between patches.  He wished to establish a partial ordering of stochastic matrices $\M$ with respect to their levels of `mixing' over which $r(\M\D)$ decreases with increased mixing. 

Variation in $(1-m)\I+m\P$ over $m$ represents variation in the incidence of a single transforming processes (such as mutation, recombination, or dispersal) that scales all transitions between states equally.  However, many natural systems have multiple transforming processes that act simultaneously, in which case the variation with respect to a single one of these processes generally takes the form $(1-m)\Q+m\P$ where $\P$ and $\Q$ are stochastic matrices.  Karlin's Theorem 5.2 does not apply for general $\Q \neq \I$.  The motivation to develop Theorem \ref {Theorem:Multivariate} \citep[Theorem 2]{Altenberg:2011:Mutation}, below, was to extend Karlin's Theorem 5.2 to processes with multiple transforming events.  

An open problem posed in \citet{Altenberg:2004:Open} and \citet{Altenberg:2012:Resolvent-Positive} is the general characterization of the matrices $\Q$, $\P$, and $\D$ such that $r([(1-m)\Q+m\P]\D)$ strictly decreases in $m$.   Theorem \ref{Theorem:Departure} \citep[Theorem 33]{Altenberg:2012:Dispersal} goes part way toward this characterization. 

%%%%%%%%%%%%%%%%%
\subsection{Kronecker Products}
A notable class of matrices that exhibit the commuting property required for Theorems \ref{Theorem:Karlin5.1}, \ref{Theorem:Departure}, \ref {Theorem:SpectralRadius}, and \ref{Theorem:Main} is the Kronecker product of powers of matrices.  Define a set of square matrices
\ab{
\Cc \eqdef \{\A_1, \A_2, \ldots, \A_L\},
}
where each $\A_i$ is an $n_i \times n_i$ matrix.  Define
\an{\label{eq:OxA}
\M(\tv) \eqdef \Ox_{i=1}^L \A_i^{t_i} 
= \A_1^{t_1} \ox \A_2^{t_2} \ox \cdots \ox \A_L^{t_L},
}
where $\ox$ is the Kronecker product (a.k.a.\ tensor product), $t_i \in \Naturals_0 \eqdef \{0, 1, 2, \ldots \}$, and $\tv \in \Naturals_0^{\ L}$.  Now define the family of such products:
\ab{
\Fc(\Cc) = \left \{ \Ox_{i=1}^L \A_i^{t_i} \suchthat t_i \in \{0, 1, 2, \ldots \} \right\}.
}
Clearly, any two members of $\Fc(\Cc)$ commute, because for any $\p, \q \in \Naturals_0^{\ L}$, 
\ab{
\M(\p) \, \M(\q) = \Ox_{i=1}^L \A_i^{p_i + q_i}= \M(\q) \, \M(\p).
}

Products of the form $\Ox_{i=1}^L \A_i^{t_i}$ arise in multivariate Markov chains for which each variate $X_i$ constitutes an independent Markov chain with transition matrix $\A_i$.  The joint Markov process is exemplified be the transmission of information in a string of $L$ symbols where transmission errors occur independently for each symbol.   Such a process includes the genetic transmission of DNA or RNA sequences with independent mutations at each site.  Under mitosis, the genome replicates approximately according to a transition matrix for a string of symbols with independent transmission errors at each site $i$:
\ab{
\M_\m \eqdef \bigotimes_{i=1}^L \left[ (1-m_i) \I_{i} + m_i \P_{i} \right  ]
=
\A[ (1-m_k) \I + m_k  \B],
}
where $m_i $ is the probability of a transforming event at site $i$, and $\P_{i}$ is the transition matrix for site $i$ given that a transforming event has occurred there.  The form $\A[ (1-m_k) \I + m_k  \B]$ is provided to show the relationship to Theorem \ref{Theorem:Main}, where $k$ may be any choice in $\{1, \ldots, L\}$, with $\A_i = (1-m_i) \I_{i} + m_i \P_{i}$ and
\an{
\A &=  \A_1 \ox \A_2 \ox \cdots \ox \A_{k-1} \ox \  \I_{k} \, \ox \A_{k+1} \ox \cdots \ox \A_L , \label{eq:Aotimes}\\
\B &= \,\,  \I_{1} \,\ox \,  \I_{2} \, \ox  \cdots \ox\  \I_{k-1} \, \ox  \P_k \ox \ \I_{k+1} \, \ox \cdots \ox \  \I_{L}.\label{eq:Botimes}
}
However, both $\A$ and $\B$ in \eqref{eq:Aotimes}  and \eqref{eq:Botimes} are reducible due to the $\I$ terms, and this somewhat alters Theorem \ref {Theorem:Main}'s condition on $\D$ for strict monotonicity of spectral radius.  This condition is seen in \eqref {eq:Ds} in the following theorem.

%%%%%%%%%%%%%%%%%%%%%%%%%%%%%%%%%%
\begin{theorem}{\pr{\cite[Theorem 2]{Altenberg:2011:Mutation}}}
\label{Theorem:Multivariate}
Consider the stochastic matrix
\an{
\label{eq:Mmuv}
\M_\m \eqdef \bigotimes_{\kappa =1}^L \left[ (1-m_\kappa ) \I_{\kappa } + m_\kappa \P_{\kappa } \right  ],
}
where each $\P_{\kappa }$ is an $n_\kappa \times n_\kappa $ transition matrix for a reversible aperiodic Markov chain, $\I_{\kappa }$ the $n_\kappa \times n_\kappa $ identity matrix, $L \geq 2$, and $\m \in (0, 1/2)^L$.   Let $\D$ be a positive $N \times N$ diagonal matrix, where $N\eqdef\prod_{\kappa =1}^L n_\kappa $.

Then for every point $\m \in (0, 1/2)^L$, the spectral radius of 
\[
\M_\m \D = \left\{ \bigotimes_{\kappa =1}^L \left[ (1-m_\kappa ) \I_{\kappa } + m_\kappa \P_{\kappa } \right ] \right\} \D
\]
is non-increasing in each $m_\kappa $.

If diagonal entries
\an{\label{eq:Ds}
 D_{\displaystyle i_1 \cdots i_\kappa \cdots i_L} , \ D_{\displaystyle i_1 \cdots i_\kappa ' \cdots i_L}
}
differ for at least one pair $i_\kappa , i_\kappa ' \in \{1, \ldots, n_\kappa \}$, for some $ i_1 \in \{1, \ldots, n_1\}$,  $\ldots$, $i_{\kappa -1} \in \{1, \ldots, n_ {\kappa -1} \}$, $i_{\kappa +1} \in \{1, \ldots, n_ {\kappa +1} \}$, $\ldots$, $i_L \in \{1, \ldots, n_L\}$,  
then
\ab{
\pmu{r(\M_\m \D)}{\kappa } < 0 .
}
\end{theorem}
%%%%%%%%
\begin{remark}\rm
The condition on the diagonal entries \eqref{eq:Ds} can be expressed simply in the cases $\kappa =1$ and $\kappa =L$, respectively, as the requirement that $\D \neq c\,\I_1 \ox \D'$ and $\D \neq \D' \ox c\,\I_L $ for any $c \in \Reals$ and any $\D'$.  For $\kappa \in \{2, \ldots, L-1\}$ similar expressions can be given by employing permutations of the tensor indices.
\end{remark}

Theorem \ref{Theorem:Multivariate} was obtained to characterize the effect of mutation rates on a clonal population, or on a gene that modifies mutation rates in a non-recombining genome.  This theorem shows that the asymptotic growth rate of an infinite population of types $\{(i_1, i_2, \ldots, i_L)\}$ is a strictly decreasing function of each mutation rate $m_ \kappa $ when the growth rates $D_i$ in \eqref{eq:Ds} differ, and non-increasing otherwise.  All the eigenvalues of $\M_\m$ are positive, as in condition C1 in Theorem \ref{Theorem:Main}, due to the assumption $m_\kappa < 1/2$ for $\kappa \in \{1, \ldots, L\}$.

The asymptotic growth rate of a quasispecies \citep{Eigen:and:Schuster:1977} at a mutation-selection balance is thus shown by Theorem \ref {Theorem:Multivariate} to be a decreasing function of the mutation rate for each base pair, a result not previously obtained with this level of generality in the multilocus mutation parameters, mutation matrices, and multilocus selection coefficients.  As a practical matter, however, in genetics $L$ may be very large, for example $L\approx 6 \times 10^{9}$ for the human genome.  For such large $L$, populations cannot exhibit the Perron vector as a stationary distribution since the population size is infinitesimal compared to the genome space of $n=4^L \approx 10^{4 \times 10^9}$.  However, in large populations models that examine a small-$L$ approximation or portion of the full genome, the Perron vector may become relevant as the stationary distribution under selection and mutation.

%%%%%%%%%%
\begin{proposition}\label{Prop:MA}
Theorem \ref {Theorem:Multivariate} extends to general symmetrizable irreducible nonnegative matrices
\ab{
\M_\m =  \bigotimes_{\kappa =1}^L \left[ (1-m_\kappa ) r(\A_{\kappa }) \, \I_{\kappa } + m_\kappa \A_{\kappa } \right ] ,
}
where each $\A_\kappa $ is a symmetrizable irreducible nonnegative $n_\kappa \times n_\kappa $ matrix.
\end{proposition}
%%%%
\begin{proof}
For any given $\A \in \{\A_\kappa\}$, let $\uv_A\tr$ be its left Perron vector, and define
\an{\label{eq:Pequiv}
\P = \frac{1}{r(\A)} \D_{\uv_A} \A \,  \D_{\uv_A}^{-1}.
}
Then $\P$ is a symmetrizable irreducible stochastic matrix:
\enumlist{
\item  $\P \geq \0$ since $r(\A) > 0$, and $\uv_A > \0$.
\item $\P$ is stochastic, since
\ab{
\ev\tr \left(\frac{1}{r(\A)} \D_{\uv_A} \A  \D_{\uv_A}^{-1} \right) = \frac{1}{r(\A)} \uv_A \tr \A \D_{\uv_A}^{-1} = \frac {r(\A)}{r(\A)} \uv_A\tr \D_{\uv_A}^{-1} = \ev\tr .
}
\item $\P$ is symmetrizable:
\ab{
\P & = \frac{1}{r(\A)} \D_{\uv_A} \A  \D_{\uv_A}^{-1}
=  \frac{1}{r(\A)} \D_{\uv_{A}} \DL \S \DR  \D_{\uv_A}^{-1} \notag
= \DL' \S \DR', \notag
}
where $\DL' = \frac{1}{r(\A)} \D_{\uv_A} \DL$, $\DR' = \DR  \D_{\uv_A}^{-1}$. 
\item $\P$ is irreducible since $[\A^t]_{ij} > 0$ if and only if 
\ab{
[\P^t]_{ij} = \frac{1}{r(\A)^t} \ u_{Ai}\  [\A^t]_{ij} \  u_{Aj}^{-1} > 0.
}
}
The spectral radius expressions in terms of $\A_\kappa $ and $\P_\kappa $ are now shown to be equivalent:
\ab{
r(\M_\m \D) 
&
= 
r \biggl( \bigotimes_{\kappa =1}^L \left[ (1-m_\kappa ) \, r(\A_{\kappa }) \, \I_{\kappa } + m_\kappa \A_{\kappa } \right ] \D \biggr)
\\
&
=
\prod_{\iloc =1}^L r(\A_ \iloc ) \ r\biggl( \bigotimes_{\kappa =1}^L \left[ (1-m_\kappa )  \I_{\kappa } + m_\kappa \frac{1}{r(\A_{\kappa })} \A_{\kappa } \right ] \D\biggr)
\\
&
=
\prod_{\iloc =1}^L r(\A_ \iloc ) \ r\biggl( \bigotimes_{\kappa =1}^L \left\{  \D_{\uv_{A\kappa }} [(1-m_\kappa )  \I_{\kappa } + m_\kappa \frac{1}{r(\A_{\kappa })} \A_{\kappa }] \D_{\uv_{A\kappa }}^{-1} \right\}  \D\biggr)
\\
&
=
\prod_{\iloc =1}^L r(\A_ \iloc ) \ r\biggl( \bigotimes_{\kappa =1}^L \left [(1-m_\kappa )  \I_{\kappa } + m_\kappa \frac{1}{r(\A_{\kappa })}\D_{\uv_{A\kappa }}  \A_{\kappa }\D_{\uv_{A\kappa }}^{-1}  \right]  \D\biggr)
\\
&
=
\prod_{\iloc =1}^L r(\A_ \iloc ) \ r\biggl( \bigotimes_{\kappa =1}^L \left [(1-m_\kappa )  \I_{\kappa } + m_\kappa \P_\kappa  \right]  \D\biggr)
\\
&
=
\prod_{\iloc =1}^L r(\A_ \iloc ) \ r( \M'_\m  \D),
}
where $\M'_\m \eqdef \bigotimes_{\kappa =1}^L \left [(1-m_\kappa )  \I_{\kappa } + m_\kappa \P_\kappa  \right]$, and each $\D_{\uv_{A\kappa }}$ is the diagonal matrix of the right Perron vector of $\A_\kappa$.  Therefore 
\an{\label{eq:MDMD}
\pf{}{m_ \kappa} r(\M_\m \D) = \prod_{\iloc =1}^L r(\A_\iloc ) \ \pf{}{m_ \kappa} r(\M'_\m \D).
}
Theorem \ref{Theorem:Multivariate}, being applicable to the right hand side of \eqref{eq:MDMD}, is thus extended to the left hand side composed of general symmetrizable irreducible nonnegative matrices.
\end{proof}

%%%%%%%%
\begin{remark}\rm
The identification $\P = \frac{1}{r(\A)} \D_{\uv_A} \A  \D_{\uv_A}^{-1}$ \eqref{eq:Pequiv} used in Proposition \ref {Prop:MA} also provides another route to extend Theorem \ref{Theorem:Departure} to Theorem \ref{Theorem:Main}, sidestepping Lemmas \ref{Lemma:Canonical} and \ref{Lemma:CanonicalFormAB}, Theorem \ref{Theorem:SpectralRadius}, and the proof of Theorem \ref{Theorem:Main}.  But these latter results are of interest in their own right and so are not omitted.
\end{remark}

%%%%%%%%%%%%%%%%%
\subsection{Temporal Properties}
Theorem \ref{Theorem:Departure} was obtained to generalize a model by McNamara and Dall \citep{McNamara:and:Dall:2011} of a population that disperses in a field of sites undergoing random change between two environments, where each environment produces its own rate of population growth.  In the generalization of \citep{McNamara:and:Dall:2011} to any number of environments \citep{Altenberg:2012:Dispersal}, environmental change is modeled as a reversible Markov chain with transition matrix $\P$, and $\Q = \lim_{t \goesto \infty} \P^t$.  The condition from Theorem \ref{Theorem:Departure} that $\P$ have all negative non-Perron eigenvalues means that the environment changes almost every time increment, whereas positive eigenvalues correspond to more moderate change.  

The correspondence originally discovered by McNamara and Dall \citep{McNamara:and:Dall:2011} was between the duration of each environment --- its \emph{sojourn time} \citep{Halmos:1949:Measurable} --- and whether natural selection was for or against dispersal.  The direction of evolution of dispersal and the sojourn times of the environment are, in the generalization of their model, both determined by conditions C1 and C2 on the signs of the non-Perron eigenvalues of the environmental change matrix \citep[Theorem 33]{Altenberg:2012:Dispersal}.  More specifically, what is determined by conditions C1 and C2 is an inequality on the \emph{harmonic mean} of the expected sojourn times of the Markov chain.  The inequality derives from a remarkably little-known identity.
%%%%%%%%%%%%%%%%%%%%%%%%%%%%%%%
\begin{lemma}[Harmonic Mean of Sojourn Times \pr{\citep[Lemma 32]{Altenberg:2012:Dispersal}}]\label{Lemma:HarmonicMean}
For a Markov chain with transition matrix $\P$, let $\tau_i(\P)$ be the expected {sojourn time} in $i$ (the mean duration of state $i$), and let $\{\lambda_i(\P)\}$ be the eigenvalues of $\P$.  Let $\EA$ and $\EH$ represent the unweighted arithmetic and harmonic means, respectively.

These are related by the following identities:
\an{
\EH(\tau_i (\P) ) \, \bigl(1-\EA(\lambda_i(\P)\bigr)  &= 1, \label{eq:HarmonicMean}
\intertext{or equivalently}
\EA(\lambda_i(\P) ) + \frac{1}{\EH(\tau_i (\P))} &=  1. \label{eq:HarmonicMeanAlt}
}
\end{lemma}

I should qualify ``little known'' --- a version of \eqref{eq:HarmonicMean} is well-known within the field of research  on social mobility, but no reference to it outside this community appears evident.  The identity arises in Shorrock's \citep{Shorrocks:1978} social mobility index
\ab{
\hat{M}(\P) 
=  \frac{1}{n-1} \sum_{i=1}^n  (1-P_{ii}),
}
where $P_{ij}$ is the probability of transition from social class $j$ to class $i$.  Shorrocks notes that $\hat{M}(\P)$ is related to the expected sojourn times (`exit times') for each class $i$, $\tau_i=1/(1-P_{ii})$, through their harmonic mean, 
\ab{
\EH(\tau_i) \eqdef \frac{1}{\displaystyle \frac{1}{n} \sum_{i=1}^n \dspfrac{1}{\tau_i} } \ . 
}
Evaluation gives
\ab{
\EH(\tau_i) = \EH\left(\frac{1}{1-P_{ii}} \right)
=
\frac{1}{\displaystyle \frac{1}{n} \sum_{i=1}^n \dspfrac{1}{1/(1-P_{ii})} }
=
\frac{n}{\displaystyle  \sum_{i=1}^n (1-P_{ii}) },
}
yielding
\ab{
\hat{M}(\P) =  \frac{1}{n-1} \sum_{i=1}^n  (1-P_{ii})
= \left(\frac{n}{n-1} \right) \frac{1}{ \EH (\tau_i) }.
}

Geweke et al. \citet{Geweke:Marshall:and:Zarkin:1986:Mobility} define another social mobility index,
\ab{
M_E(\P) 
= 
\frac{n - \sum_{i=1}^n | \lambda_i(\P)|}{n-1}.
}
They note that when all the eigenvalues of $\P$ are real and nonnegative, then $\hat{M}(\P) = M_E(\P)$, by the trace identity $\sum_{i=1}^n P_{ii} = \sum_{i=1}^n \lambda_i(\P)$.   Numerous papers cite this correspondence \citep{Quah:1996:Aggregate,Redding:2002:Specialization}.  However no expression of the identity in terms of the harmonic and arithmetic means, as in the forms \eqref{eq:HarmonicMean} or \eqref{eq:HarmonicMeanAlt}, is evident.

Next, the eigenvalue conditions C1 and C2 are applied to the identity \eqref{eq:HarmonicMean}.

%%%%%%%%%%
\begin{theorem}[From \pr{\citet[Theorem 33]{Altenberg:2012:Dispersal}}]\label{Theorem:HarmonicTime}
Let $\P$ be the $n \times n$ transition matrix of an irreducible Markov chain whose eigenvalues are real.  Let $\tau_i (\P) = 1/(1 - P_{ii})$ be the expected sojourn time in state $i$. 
\enumlist{
\item[C1.] If all eigenvalues of $\P$ are positive, then
\an{
\EH(\tau_i (\P) )  > 1 + \frac{1}{n{-}1}.\label{eq:EH>}
}
\item[C2.] If all non-Perron eigenvalues of $\P$ are negative, then
\an{
1 \leq \EH(\tau_i (\P) )  < 1 + \frac{1}{n{-}1}.\label{eq:EH<}
}
\item[C3.] If all non-Perron eigenvalues of $\P$ are zero, then
\an{
1 \leq \EH(\tau_i (\P) )  = 1 + \frac{1}{n{-}1}.\label{eq:EH=}
}
\item[C4.]   If all non-Perron eigenvalues of $\P$ are the same sign or zero, and at least one is nonzero, then inequalities \eqref {eq:EH>} and \eqref {eq:EH<} are unchanged.
}	% enumlist
\end{theorem}
%%%%
\begin{proof}
 The following inequalities are readily seen to be equivalent:
\an{
\EH(\tau_i) & = \frac{1}{\displaystyle \frac{1}{n} \sum_{i=1}^n \dspfrac{1}{\tau_i} } > 1 + \frac{1}{n{-}1} =\frac{n}{n{-}1} \label{eq:EHmoreV2} \\ 
\iff  & n{-}1 > \sum_{i=1}^n  \frac{1}{\tau_i} = \sum_{i=1}^n  \frac{1}{1/(1-P_{ii})} 
= n - \sum_{i=1}^n {P}_{ii} \notag\\
\iff  &1  <   \sum_{i=1}^n {P}_{ii} = \sum_{i=1}^n \lambda_i(\P)  = 1 +  \sum_{i=2}^n \lambda_i(\P) \notag\\
\iff  & 0  <   \sum_{i=2}^n \lambda_i(\P) \label{eq:EHLambda}.
}
The analogous equivalence holds if the directions of the inequalities are reversed.  If $\lambda_i(\P) > 0$ for all $i$ then  \eqref{eq:EHLambda}, \eqref{eq:EHmoreV2}, and \eqref{eq:EH>} hold.  Conversely, 
if $\lambda_i(\P) < 0$ for $i = 2, \ldots, n$ then $\sum_{i=2}^n \lambda_i(\P) < 0$, reversing the direction of the inequalities, and the right side of \eqref{eq:EH<} holds;  the left side of \eqref{eq:EH<} clearly holds since $\tau_i (\P) = 1/(1 - P_{ii}) \geq 1$ for each $i$.  If $\lambda_i(\P) = 0$ for $i = 2, \ldots, n$ then $\EH(\tau_i)=1 + {1}/(n{-}1)$.  If $\lambda_i(\P) \geq 0$ for $i \in \{2, \ldots, n\}$, and $\lambda_i(\P) > 0$ for some $i \in \{2, \ldots, n\}$, then $ \sum_{i=2}^n \lambda_i(\P) > 0$ so \eqref {eq:EHLambda} continues to hold; analogously for the reverse inequality.
\end{proof}

We have seen that conditions C1 and C2 are sufficient to determine opposite directions of inequality in two very different expressions, one involving the temporal behavior of a Markov chain, $\EH(\tau_i(\P)) > 1 + 1/(n-1)$ (under  condition C1), and the other involving the interaction of the chain with heterogeneous growth rates, $\dfinline{r( \P[(1-m)\I+m\Q]\D)}{m} < 0$ and $r(\P\D) > r(\P^2\D)$ (with reverse directions under C2).

The inference in these results goes from the eigenvalue sign conditions, C1 and C2, to the inequalities.  The converse, an implication from the inequality directions to the eigenvalue sign conditions, is found only in the case $n=2$.  It would be of empirical interest to know if there exist classes of stochastic matrices $\P$ for $n \geq 3$ in which the temporal behavior has direct implications upon the spectral radius, i.e.\ $\EH(\tau_i(\P))$ tells us about $\dfinline{r( \P[(1-m)\I+m\Q]\D)}{m}$ and $ r(\P^2\D)/r(\P\D)$, or vice versa, without recourse to conditions C1 and C2.  

An example of such a class for $n \geq 3$ is devised using rank-one matrices.   Let $\Pc_n$ be the set of probability vectors of length $n$, so $\evt \x = 1, \x \geq \0$ for  $\x \in \Pc_n$.  Define the set of stochastic matrices
\ab{
\Rc_n \eqdef \setd{ (1-\alpha) \I + \alpha \vv \evt \suchthat \vv \in \Pc_n, \vv > \0, \alpha \in \bigl(0, \min_i{\frac{1}{1-v_i}}\bigr] }.
}
The upper bound on $\alpha$ allows $\alpha \geq 1$ while assuring $1-\alpha+ \alpha v_i \geq 0$ for each $i$, so that $(1-\alpha) \I + \alpha \vv \evt$ is nonnegative. 

\begin{corollary}
Let $\P \in \Rc_n$, let $\Q$ be a symmetrizable irreducible stochastic matrix that commutes with $\P$, and let $\D$ be an $n \times n$ nonscalar positive diagonal matrix.  Then 
\ab{
\df{}{m}r( \P[(1-m)\I + m \Q] \D) & < 0  \text{ if and only if } \EH(\tau_i (\P) )  > 1 + \frac{1}{n{-}1},\\
\df{}{m}r( \P[(1-m)\I + m \Q] \D) &= 0\text{ if and only if } \EH(\tau_i (\P) )  = 1 + \frac{1}{n{-}1},
\intertext{and}
\df{}{m}r( \P[(1-m)\I + m \Q] \D) & > 0 \text{ if and only if } 
1 \leq \EH(\tau_i (\P) )  < 1 + \frac{1}{n{-}1}.
}
\end{corollary}

\begin{corollary}
Let $\P \in \Rc_n$, and let $\D$ be an $n \times n$ nonscalar positive diagonal matrix.  Then 
\ab{
r(\P^2 \D) &< r(\P\D) \text{ if and only if } \EH(\tau_i (\P) )  > 1 + \frac{1}{n{-}1},\\
r(\P^2 \D)& = r(\P\D) \text{ if and only if } \EH(\tau_i (\P) )  = 1 + \frac{1}{n{-}1},
\intertext{and}
r(\P^2 \D) &> r(\P\D) \text{ if and only if } 
1 \leq \EH(\tau_i (\P) )  < 1 + \frac{1}{n{-}1}.
}
\end{corollary}
\begin{proof}
Any $\P \in \Rc_n$ is irreducible since by hypothesis $\vv > \0$, $\alpha > 0$.   To apply Theorem \ref{Theorem:Main}, we must verify that $\P \in \Rc_n$ is symmetrizable:
\ab{
\P &= (1-\alpha) \I + \alpha \vv \evt =  \D_{\vv}^{1/2}[ (1-\alpha) \I +  \alpha (\D_{\vv}^{1/2} \ev \evt \D_{\vv}^{1/2})]  \D_{\vv}^{-1/2}.
}

Let $\z_i$ be a right eigenvector of $\P \in \Rc_n$ associated with $\lambda_i(\P)$.  Then
\ab{
\lambda_i \z_i = \P \z_i &
= [ (1-\alpha) \I + \alpha \vv \evt]\z_i
=  (1-\alpha) \z_i + \alpha \vv \evt \z_i
\\
\iff & (\lambda_i  - 1 + \alpha) \z_i
= \alpha \vv \evt \z_i 
\\
\iff & \evt \z_i 
= 0, \lambda_i = 1-\alpha \text{ \  or \ } \z_i = \frac{\alpha (\evt \z_i)}{\lambda_i + \alpha - 1} \vv. 
}
The upper bound on $\alpha$ used to define $\Rc_n$ also gives
\an{\label{eq:1-alpha}
1-\alpha \in \bigl[1- \min_i{\frac{1}{1-v_i}}, 1 \big)
= \bigl[- \min_i{\frac{v_i}{1-v_i}}, 1 \big).
}
So either (1) $\lambda_i = 1 - \alpha \in [- \min_j{\frac{v_j}{1-v_j}}, 1)$, by \eqref{eq:1-alpha}, or (2) $\z_i$ is proportional to the right Perron vector of $\P$, which has $\lambda_i=1$, hence $\z_i = \fracinline{\alpha (\evt \z_i)}{\alpha}\ \vv = c \vv$.  Thus all of the non-Perron eigenvalues of $\P$ equal $1-\alpha$, and may be either positive, zero, or negative in the range $ [- \min_i{\frac{v_i}{1-v_i}}, 0)$, in which case exactly one of conditions C1, C3, or C2 is met, respectively, for Theorems \ref{Theorem:Main} and \ref{Theorem:HarmonicTime} and Corollary \ref {Corollary:A2D}, with the consequent implications.
\end{proof}

%%%%%%%%%%%%%%%%
\subsection{Other Applications}

Condition C2 is met by nonnegative \emph{conditionally negative definite} matrices \citep[Chapter 4]{Bapat:and:Raghavan:1997}, \citep{Bhatia:2007:Positive}.  Symmetric conditionally negative definite matrices arise in the analysis of the one-locus, multiple-allele viability selection model.  If the matrix of fitness coefficients $\W$ allows the existence of a polymorphism with all alleles present, then the polymorphism is globally stable if $\W$ is conditionally negative definite \citep{Kingman:1961:Mathematical}  (Kingman's exact condition  being that they need only be conditionally negative semidefinite).

%%%%%%%%%%%%%
\section{Open Problems}

The conditions in Theorem \ref{Theorem:Main} that all the eigenvalues of $\A$ be positive (C1), or that all the non-Perron eigenvalues be negative (C2), are clearly very strong, and leave us with no results for intermediate conditions.  Such results are likely to be had by placing additional conditions on the matrices $\E$, $\K$, and $\D$, but this remains an unexplored area.  

The condition of symmetrizability imposes a large constraint on the generality of the results here.  For non-symmetrizable matrices, we lose use of the Rayleigh-Ritz variational formula, and the spectral radius no longer has the sum-of-squares representation \eqref{eq:rho-y2}, which is our principal tool.  It is an open question how many of these results extend to general, non-symmetrizable nonnegative matrices.

%%%%%%%%%%%%%%%%
\section*{Acknowledgements}
I thank Joel E. Cohen for bringing to my attention the question of $r(\A) \, r(\A\D)$ vs. $r(\A^2\D)$, and Shmuel Friedland for referring Joel Cohen to me.  I thank an anonymous reviewer for pointing out Levinger's Theorem and \citet{Fiedler:1995:Numerical}.  I thank Laura Marie Herrmann for assistance with the literature search.
%%%%%%%%%%%%%%%%%%%%%%%%%
%\section*{References}	% 
%\input{/usr/local/LATEX/biblist.tex}


\begin{thebibliography}{10}

\bibitem{Afriat:1974}
S.~Afriat.
\newblock On sum-symmetric matrices.
\newblock {\em Linear Algebra and its Applications}, 8(2):129--140, 1974.

\bibitem{Alpin:and:Kolotilina:1998}
Y.~A. Alpin and L.~Y. Kolotilina.
\newblock Inequalities for the perron root related to levinger's theorem.
\newblock {\em Linear algebra and its applications}, 283(1):99--113, 1998.

\bibitem{Altenberg:2004:Open}
L.~Altenberg.
\newblock Open problems in the spectral analysis of evolutionary dynamics.
\newblock In A.~Menon, editor, {\em Frontiers of Evolutionary Computation},
  volume~11 of {\em Genetic Algorithms And Evolutionary Computation Series},
  pages 73--102. Kluwer Academic Publishers, Boston, MA, 2004.

\bibitem{Altenberg:2011:Mutation}
L.~Altenberg.
\newblock An evolutionary reduction principle for mutation rates at multiple
  loci.
\newblock {\em Bulletin of Mathematical Biology}, 73:1227--1270, 2011.

\bibitem{Altenberg:2012:Dispersal}
L.~Altenberg.
\newblock The evolution of dispersal in random environments and the principle
  of partial control.
\newblock {\em Ecological Monographs}, 82(3):297--333, 2012.

\bibitem{Altenberg:2012:Resolvent-Positive}
L.~Altenberg.
\newblock Resolvent positive linear operators exhibit the reduction phenomenon.
\newblock {\em Proceedings of the National Academy of Sciences U.S.A.},
  109(10):3705--3710, 2012.

\bibitem{Altenberg:2013:Sharpened}
L.~Altenberg.
\newblock A sharpened condition for strict log-convexity of the spectral radius
  via the bipartite graph.
\newblock {\em Linear Algebra and Its Applications}, 438(9):3702 -- 3718, 2013.

\bibitem{Bapat:and:Raghavan:1997}
R.~B. Bapat and T.~E.~S. Raghavan.
\newblock {\em Nonnegative Matrices and Applications}.
\newblock Cambridge University Press, Cambridge, UK, 1997.

\bibitem{Bhatia:2007:Positive}
R.~Bhatia.
\newblock {\em Positive Definite Matrices}.
\newblock Princeton Series in Applied Mathematics. Princeton University Press,
  Princeton, NJ, 2007.

\bibitem{Caswell:2000}
H.~Caswell.
\newblock {\em Matrix Population Models}.
\newblock Sinauer Associates, Sunderland, MA, 2nd edition, 2000.

\bibitem{Cohen:2013:Cauchy}
J.~E. Cohen.
\newblock Cauchy inequalities for the spectral radius of products of diagonal
  and nonnegative matrices.
\newblock {\em Proceedings of the American Mathematical Society}, In press,
  2012.

\bibitem{Cohen:Friedland:Kato:and:Kelly:1982}
J.~E. Cohen, S.~Friedland, T.~Kato, and F.~P. Kelly.
\newblock Eigenvalue inequalities for products of matrix exponentials.
\newblock {\em Linear Algebra and Its Applications}, 45:55--95, 1982.

\bibitem{Eaves:Hoffman:Rothblum:and:Schneider:1985}
B.~Eaves, A.~Hoffman, U.~Rothblum, and H.~Schneider.
\newblock Line-sum-symmetric scalings of square nonnegative matrices.
\newblock {\em Mathematical Programming Study}, 25(Essays in Honor of George B.
  Dantzig Part {II}):124--141, 1985.

\bibitem{Eigen:and:Schuster:1977}
M.~Eigen and P.~Schuster.
\newblock The hypercycle: A principle of natural self-organization.
\newblock {\em Naturwissenschaften}, 64:541--565, 1977.

\bibitem{Elsner:Johnson:and:DaSilva:1988}
L.~Elsner, C.~R. Johnson, and J.~Dias~da Silva.
\newblock The perron root of a weighted geometric mean of nonneagative
  matrices.
\newblock {\em Linear and Multilinear Algebra}, 24(1):1--13, 1988.

\bibitem{Fiedler:1995:Numerical}
M.~Fiedler.
\newblock Numerical range of matrices and levinger's theorem.
\newblock {\em Linear Algebra and Its Applications}, 220:171--180, 1995.

\bibitem{Fiedler:Johnson:Markham:and:Neumann:1985}
M.~Fiedler, C.~R. Johnson, T.~L. Markham, and M.~Neumann.
\newblock A trace inequality for {M}-matrices and the symmetrizability of a
  real matrix by a positive diagonal matrix.
\newblock {\em Linear Algebra and Its Applications}, 71:81--94, 1985.

\bibitem{Friedland:and:Karlin:1975}
S.~Friedland and S.~Karlin.
\newblock Some inequalities for the spectral radius of non-negative matrices
  and applications.
\newblock {\em Duke Mathematical Journal}, 42(3):459--490, 1975.

\bibitem{Geweke:Marshall:and:Zarkin:1986:Mobility}
J.~Geweke, R.~Marshall, and G.~Zarkin.
\newblock Mobility indices in continuous time {Markov} chains.
\newblock {\em Econometrica}, 54(6):1407--1423, 1986.

\bibitem{Gould:1966:Variational}
S.~H. Gould.
\newblock {\em Variational Methods for Eigenvalue Problems: {A}n Introduction
  to the Weinstein Method of Intermediate Problems}, volume~10.
\newblock University of Toronto Press; London: Oxford University Press, 1966.

\bibitem{Halmos:1949:Measurable}
P.~R. Halmos.
\newblock Measurable transformations.
\newblock {\em Bulletin of the American Mathematical Society},
  55(11):1015--1034, 1949.

\bibitem{Horn:and:Johnson:1985}
R.~A. Horn and C.~R. Johnson.
\newblock {\em Matrix Analysis}.
\newblock Cambridge University Press, Cambridge, 1985.

\bibitem{Johnson:1977:Hadamard}
C.~R. Johnson.
\newblock A hadamard product involving n-matrices.
\newblock {\em Linear and Multilinear Algebra}, 4(4):261--264, 1977.

\bibitem{Karlin:1976}
S.~Karlin.
\newblock Population subdivision and selection migration interaction.
\newblock In S.~Karlin and E.~Nevo, editors, {\em Population Genetics and
  Ecology}, pages 616--657. Academic Press, New York, 1976.

\bibitem{Karlin:1982}
S.~Karlin.
\newblock Classifications of selection--migration structures and conditions for
  a protected polymorphism.
\newblock In M.~K. Hecht, B.~Wallace, and G.~T. Prance, editors, {\em
  Evolutionary Biology}, volume~14, pages 61--204. Plenum Publishing
  Corporation, New York, 1982.

\bibitem{Keilson:1979}
J.~Keilson.
\newblock {\em {{Markov}} Chain Models: Rarity and Exponentiality}.
\newblock Springer-Verlag, New York, 1979.

\bibitem{Kingman:1961:Mathematical}
J.~F.~C. Kingman.
\newblock A mathematical problem in population genetics.
\newblock {\em Mathematical Proceedings of the Cambridge Philosophical
  Society}, 57:574--582, 1961.

\bibitem{Kirkland:Li:and:Schreiber:2006}
S.~Kirkland, C.~K. Li, and S.~J. Schreiber.
\newblock On the evolution of dispersal in patchy landscapes.
\newblock {\em {SIAM} Journal on Applied Mathematics}, 66:1366--1382, 2006.

\bibitem{Levinger:1970}
B.~W. Levinger.
\newblock An inequality for nonnegative matrices.
\newblock {\em Notices of the American Mathematical Society}, 17, 1970.

\bibitem{McNamara:and:Dall:2011}
J.~M. McNamara and S.~R. Dall.
\newblock The evolution of unconditional strategies via the `multiplier
  effect'.
\newblock {\em Ecology Letters}, 14(3):237--243, Mar. 2011.

\bibitem{Nussbaum:1986:Convexity}
R.~D. Nussbaum.
\newblock Convexity and log convexity for the spectral radius.
\newblock {\em Linear Algebra and Its Applications}, 73:59--122, 1986.

\bibitem{Quah:1996:Aggregate}
D.~Quah.
\newblock Aggregate and regional disaggregate fluctuations.
\newblock {\em Empirical Economics}, 21(1):137--159, 1996.

\bibitem{Redding:2002:Specialization}
S.~Redding.
\newblock Specialization dynamics.
\newblock {\em Journal of International Economics}, 58(2):299--334, 2002.

\bibitem{Seneta:2006}
E.~Seneta.
\newblock {\em Non-negative Matrices and {{Markov}} Chains}.
\newblock Springer-Verlag, New York, revised edition, 2006.

\bibitem{Shorrocks:1978}
A.~F. Shorrocks.
\newblock The measurement of mobility.
\newblock {\em Econometrica}, 46(5):1013--1024, 1978.

\end{thebibliography}
\end{document}